\documentclass[11pt]{amsart}
\usepackage{amsmath, amssymb, amscd, mathrsfs, url, pinlabel,verbatim}
\usepackage[pagebackref]{hyperref}
\usepackage[margin=1.25in,marginparwidth=0.75in,centering,letterpaper,dvips]{geometry}
\usepackage{color,dcpic,latexsym,graphicx,epstopdf,comment}
\usepackage[all]{xy}
\usepackage[dvipsnames]{xcolor}
\usepackage{tikz,tikz-cd,pgfplots}
\usepackage{upquote,tabularx,textcomp}
\usepackage[shortlabels]{enumitem}
\usepackage[color=blue!20!white,textsize=tiny]{todonotes}

\title{Lens space surgeries on knots of small genus}

\author[Misha Schmalian]{Misha Schmalian}
\address{Department of Mathematics\\Imperial College London}
\email{m.schmalian@imperial.ac.uk}

\author[Steven Sivek]{Steven Sivek}
\address{Department of Mathematics\\Imperial College London}
\email{s.sivek@imperial.ac.uk}

\makeatletter
\newtheorem*{rep@theorem}{\rep@title}
\newcommand{\newreptheorem}[2]{%
\newenvironment{rep#1}[1]{%
 \def\rep@title{#2 \ref{##1}}%
 \begin{rep@theorem}}%
 {\end{rep@theorem}}}
\makeatother

\newtheorem {theorem}{Theorem}
\newreptheorem{theorem}{Theorem}
\newtheorem {lemma}[theorem]{Lemma}
\newtheorem {proposition}[theorem]{Proposition}
\newtheorem {corollary}[theorem]{Corollary}
\newtheorem {conjecture}[theorem]{Conjecture}

\numberwithin{equation}{section}
\numberwithin{theorem}{section}

\theoremstyle{definition}

\newtheorem{remark}[theorem]{Remark}
\newtheorem*{remark*}{Remark}

\newtheorem{example}[theorem]{Example}

\newlist{pcases}{enumerate}{1}
\setlist[pcases]{
  label=\bf{Case~\arabic*:}\protect\thiscase.~,
  ref=\arabic*,
  align=left,
  labelsep=0pt,
  leftmargin=0pt,
  labelwidth=0pt,
  parsep=0pt
}
\newcommand{\case}[1][]{%
  \if\relax\detokenize{#1}\relax
    \def\thiscase{}%
  \else
    \def\thiscase{~#1}%
  \fi
  \item
}

\newcommand{\Z}{\mathbb{Z}}

\newcommand{\R}{\mathbb{R}}

\newcommand{\F}{\mathbb{F}}
\newcommand{\Q}{\mathbb{Q}}

\newcommand{\cF}{\mathcal{F}}

\newcommand\hfk{\mathit{HFK}}
\newcommand\hfkhat{\widehat{\hfk}}

\DeclareFontFamily{U}{mathx}{\hyphenchar\font45}
\DeclareFontShape{U}{mathx}{m}{n}{
      <5> <6> <7> <8> <9> <10>
      <10.95> <12> <14.4> <17.28> <20.74> <24.88>
      mathx10
      }{}
\DeclareSymbolFont{mathx}{U}{mathx}{m}{n}
\DeclareFontSubstitution{U}{mathx}{m}{n}
\DeclareMathAccent{\widecheck}{0}{mathx}{"71}

\newcommand{\hfhat}{\widehat{\mathit{HF}}}

\makeatletter
\DeclareFontFamily{OMX}{MnSymbolE}{}
\DeclareSymbolFont{MnLargeSymbols}{OMX}{MnSymbolE}{m}{n}
\SetSymbolFont{MnLargeSymbols}{bold}{OMX}{MnSymbolE}{b}{n}
\DeclareFontShape{OMX}{MnSymbolE}{m}{n}{
    <-6>  MnSymbolE5
   <6-7>  MnSymbolE6
   <7-8>  MnSymbolE7
   <8-9>  MnSymbolE8
   <9-10> MnSymbolE9
  <10-12> MnSymbolE10
  <12->   MnSymbolE12
}{}
\DeclareFontShape{OMX}{MnSymbolE}{b}{n}{
    <-6>  MnSymbolE-Bold5
   <6-7>  MnSymbolE-Bold6
   <7-8>  MnSymbolE-Bold7
   <8-9>  MnSymbolE-Bold8
   <9-10> MnSymbolE-Bold9
  <10-12> MnSymbolE-Bold10
  <12->   MnSymbolE-Bold12
}{}

\let\llangle\@undefined
\let\rrangle\@undefined
\DeclareMathDelimiter{\llangle}{\mathopen}%
                     {MnLargeSymbols}{'164}{MnLargeSymbols}{'164}
\DeclareMathDelimiter{\rrangle}{\mathclose}%
                     {MnLargeSymbols}{'171}{MnLargeSymbols}{'171}
\makeatother

\newcounter{desccount}

\newcommand{\descref}[1]{\hyperref[#1]{#1}}

\usetikzlibrary{calc,intersections}
\tikzset{every picture/.style=thick}
\tikzset{link/.style = { white, double = black, line width = 1.75pt, double distance = 1.25pt, looseness=1.75 }}
\tikzset{crossing/.style = {draw, circle, dotted, minimum size=0.5cm, inner sep=0, outer sep=0}}
\pgfplotsset{compat=1.12}

\begin{document}

\begin{abstract}
We prove that the $(-2,3,7)$ pretzel knot is the only hyperbolic knot of genus at most $5$ with a lens space surgery.  Our key technical result asserts that if a genus-$g$ hyperbolic knot $K$ has an elliptic surgery of slope greater than $4g-3$, then the invariant foliations of its monodromy have an orientation that is reversed by the monodromy, and so the opposite of the dilatation must be a root of the Alexander polynomial of $K$.  Combined with prior work of Moser, Wu, Bleiler--Litherland, Baker, and Greene, this verifies the Berge conjecture for knots of genus at most $5$.
\end{abstract}

\maketitle

\section{Introduction}

Let $K \subset S^3$ be a nontrivial knot on which some Dehn surgery produces a lens space, say of slope $r \in \Q$.  The cyclic surgery theorem \cite{cgls} asserts that if $K$ is not a torus knot, then $r$ must be an integer.  Around 1990, Berge \cite{berge} observed that \emph{doubly primitive knots} have lens space surgeries of integral slope; he described twelve such families of knots, now called \emph{Berge knots}, and conjectured that every doubly primitive knot belongs to one of these families.  In fact, it is expected that all integer lens space surgeries arise from Berge's examples.  This is now known as the Berge conjecture:

\begin{conjecture} \label{conj:berge}
Suppose that $K \subset S^3$ has a lens space surgery of some integral slope $n$.  Then $K$ is doubly primitive, and the pair $(K,n)$ arises from Berge's construction.
\end{conjecture}

Conjecture~\ref{conj:berge} remains open, but Greene \cite{greene-lens} showed that Berge's construction accounts for every lens space realized by an integral surgery, and moreover that every doubly primitive knot is a Berge knot.  (The conjecture is now known for all tunnel number one knots, which strictly contain the doubly primitive knots, by recent work of Li, Moriah, and Pinsky \cite{li-moriah-pinsky}.)  Lens space surgeries on torus knots and on satellite knots were classified by Moser \cite{moser} and by Wu \cite{wu-cyclic} and Bleiler--Litherland \cite{bleiler-litherland}, respectively, so we can restrict our attention to hyperbolic knots.  As the main result of this paper, we verify Berge's conjecture for knots of genus at most five.

\begin{theorem} \label{thm:main}
Let $K$ be a hyperbolic knot of genus at most $5$, and suppose that $K$ admits a lens space surgery of positive slope.  Then $K$ is the $(-2,3,7)$ pretzel knot.
\end{theorem}

The $(-2,3,7)$ pretzel knot has two lens space surgeries, of slopes 18 and 19 \cite[\S4]{fs-lens}.

Theorem~\ref{thm:main} nearly completes the proof of the following conjecture of Goda and Teragaito \cite{goda-teragaito}, which was directly motivated by Berge's tabulation.

\begin{conjecture}[{\cite{goda-teragaito}}] \label{conj:goda-teragaito}
Let $K \subset S^3$ be a hyperbolic knot of genus $g \geq 1$, and suppose that $r$-surgery on $K$ is a lens space for some $r \geq 0$.  Then $K$ is fibered, and $2g + 8 \leq r \leq 4g - 1$.
\end{conjecture}

To be precise, Ghiggini \cite{ghiggini} and Ni \cite[Corollary~1.3]{ni-hfk} proved that $K$ must be fibered, and Baker \cite[Theorem~1.2]{baker-smallgenus} established the upper bound by proving that if $r \geq 4g-1$ then $K$ is a Berge knot.  (In fact, Baker's work already implies the Berge conjecture for knots of genus at most three.)  For the lower bound, Greene \cite[Theorem~1.4]{greene-lens} proved that $r \geq 2g+8$ with exactly two possible classes of exceptions, as enumerated in terms of their knot Floer homology in \cite[\S5.2]{greene-cabling}:
\begin{enumerate}
\item $(r,g) = (14,4)$ and $\hfkhat(K) \cong \hfkhat(T_{3,5})$, with $S^3_{14}(K) \cong S^3_{14}(T_{3,5})$.
\item $(r,g) = (19,6)$ and $\hfkhat(K) \cong \hfkhat(T_{4,5})$, with $S^3_{19}(K) \cong S^3_{19}(T_{4,5})$.
\end{enumerate}
Theorem~\ref{thm:main} eliminates the first possibility, and it also answers a question of Bleiler and Litherland \cite{bleiler-litherland} as follows.

\begin{corollary} \label{cor:bl}
Let $L$ be a lens space of order $|\pi_1(L)| < 18$.  Then $L$ cannot be obtained by Dehn surgery on a non-torus knot.
\end{corollary}

We also observe that Theorem~\ref{thm:main} implies that 18 is a \emph{characterizing slope} for $P(-2,3,7)$. 

\begin{corollary} \label{cor:characterizing-18}
If $K \subset S^3$ is a knot satisfying $S^3_{18}(K) \cong S^3_{18}\big(P(-2,3,7)\big)$ as oriented 3-manifolds, then $K = P(-2,3,7)$.
\end{corollary}

Theorem~\ref{thm:main} also shows that 19 is a characterizing slope for $P(-2,3,7)$, but this already followed from Baker's work.

\subsection*{Outline of the proof}

We establish Theorem~\ref{thm:main} as a consequence of the following.

\begin{theorem} \label{thm:large-elliptic-surgery-main}
Let $K \subset S^3$ be a hyperbolic knot of genus $g \leq 5$ on which $r$-surgery is an elliptic manifold for some $r > 4g-3$.  Then $K = P(-2,3,7)$.
\end{theorem}

We first explain how Theorem~\ref{thm:main} follows quickly from this theorem.

\begin{proof}[Proof of Theorem~\ref{thm:main}]
By Greene's aforementioned work on the Goda--Teragaito conjecture, it suffices to consider the case where $K$ is hyperbolic, of some genus $g \leq 5$, and the lens space surgery has slope $r$ satisfying either $(r,g) = (14,4)$ or $r \geq 2g+8$.  Theorem~\ref{thm:large-elliptic-surgery-main} says that in this case, if $K \neq P(-2,3,7)$ then we must have $r \leq 4g-3$.  This eliminates the possibility that $(r,g) = (14,4)$, and we are left with
\[ 2g+8 \leq r \leq 4g-3, \]
or equivalently $2g \geq 11$, which is a contradiction.
\end{proof}

Now we outline the proof of Theorem~\ref{thm:large-elliptic-surgery-main}.  Elliptic 3-manifolds are Heegaard Floer L-spaces, meaning that they have the simplest possible Heegaard Floer homology \cite{osz-lens}, so any knot $K$ with an elliptic surgery must be fibered \cite{ghiggini, ni-hfk}.  If $K$ is hyperbolic of genus $g \geq 1$, then its monodromy is freely isotopic to a pseudo-Anosov map
\[ \phi: \Sigma_{g,1} \to \Sigma_{g,1}, \]
whose suspension flow on $S^3 \setminus N(K)$ has some degeneracy slope $d$.  Gabai and Oertel \cite{gabai-oertel} showed that if $r$-surgery on $K$ produces an elliptic 3-manifold, then $r$ must be very close to $d$; here the assumption that $r > 4g-3$ lets us conclude that $d$ is either $4g-3$ or $4g-2$.  (This argument also yields a self-contained proof of the upper bound $r \leq 4g-1$, in Proposition~\ref{prop:upper-bound}.)  We rule out $d=4g-3$ by showing that it would force $\phi$ to have a fixed point, which cannot happen for knots with L-space surgeries \cite{ni-fixed-note, ghiggini-spano}.

The map $\phi$ fixes a pair of transversely measured foliations $\cF_s$ and $\cF_u$ on $\Sigma_{g,1}$, rescaling their transverse measures by a factor of $\lambda^{\pm1}$ where $\lambda > 1$ is the dilatation of $\phi$.  Once we know that $d=4g-2$, it follows that for example $\cF_s$ must be nonsingular, and then its first Stiefel--Whitney class is preserved by $\phi^*$.  We argue that $w_1(\cF_s)$ must be zero, because otherwise the kernel of
\[ 1 - \phi^* : H^1(\Sigma_{g,1}; \Z/2\Z) \to H^1(\Sigma_{g,1}; \Z/2\Z) \]
is nonzero, and this would imply that $\Delta_K(1) = 1$ is zero mod $2$.  Thus $\cF_s$ is orientable, and then either $\lambda$ or $-\lambda$ is a root of the Alexander polynomial $\Delta_K(t)$.  (See Proposition~\ref{lem:even-prong-dilatation} for a more general statement.) But $\Delta_K(t)$ is strongly constrained for L-space knots \cite{osz-lens, krcatovich}, and as long as $g \leq 5$ this forces $\Delta_K(t)$ to agree with the Alexander polynomial of $P(-2,3,7)$.  We deduce that the monodromies of $K$ and of $P(-2,3,7)$ have the same dilatation, and work of Tsang \cite{tsang-dilatation} lets us conclude that they are equivalent, hence $K = P(-2,3,7)$.

\begin{remark}
The proof of Theorem~\ref{thm:large-elliptic-surgery-main} only uses the hypothesis $g \leq 5$ to pin down the Alexander polynomial of $K$.  For larger genera, Corollary~\ref{cor:alexander-dilatation} still guarantees that if $K$ has an elliptic surgery of slope $r > 4g-3$ then $\Delta_K(-\lambda) = 0$.  The same argument shows that any genus-6 knot with a finite surgery of slope $22$ or $23$ must have the same dilatation as the pretzel knot $P(-2,3,9)$, which has finite, non-cyclic surgeries of both slopes \cite{bleiler-hodgson}; see Example~\ref{ex:p-239}.
\end{remark}

\begin{remark}
Baker, Kegel, and McCoy \cite[Corollary~1.13]{baker-kegel-mccoy} recently proved that if $K$ is a genus-$g$ knot other than $T_{2,2g+1}$ and some $r$-surgery on $K$ is a lens space, then $r \leq 3g+4$.  This improves the upper bound of Conjecture~\ref{conj:goda-teragaito} for all $g > 5$ well beyond what the techniques in this paper can accomplish.
\end{remark}

\subsection*{Organization}

Section~\ref{sec:l-space} contains a review of relevant facts from Heegaard Floer homology about L-space knots, which include all knots with elliptic surgeries.  We then prove Theorem~\ref{thm:large-elliptic-surgery-main} in Section~\ref{sec:monodromy}.

\subsection*{Acknowledgments}

Qilong Guo \cite{guo} independently announced a proof of Corollary~\ref{cor:bl} by essentially the same argument, with the proof appearing on the arXiv the day before we submitted this article.  We had hoped to develop the ideas in this article further, but were compelled to release it sooner than we would have liked.

The authors were supported by the Engineering and Physical Sciences Research Council [grant number UKRI1016]. No data were created or analyzed in this work.   No AI was used to prepare either the ideas or the text in this article.

\section{Background on L-space knots} \label{sec:l-space}

An \emph{L-space} is a rational homology 3-sphere $Y$ whose Heegaard Floer homology satisfies
\[ \hfhat(Y) \cong \Z^{|H_1(Y)|}. \]
Lens spaces are L-spaces, as are many other 3-manifolds.

\begin{theorem}[{\cite[Proposition~2.3]{osz-lens}}] \label{thm:elliptic}
Let $Y$ be a 3-manifold with elliptic geometry.  Then $Y$ is an L-space.
\end{theorem}

We say a knot $K \subset S^3$ is an \emph{L-space knot} if some rational surgery of positive slope $r > 0$ is an L-space.  L-space knots are known to satisfy several strong restrictions, some of which we list here.

\begin{theorem} \label{thm:lspace-alexander}
If $K$ is an L-space knot of genus $g \geq 1$, then its Alexander polynomial has the form
\[ \Delta_K(t) = t^g - t^{g-1} + \dots - t^{1-g} + t^{-g}, \]
where the nonzero coefficients of $\Delta_K(t)$ are all $\pm1$ and alternate in sign.
\end{theorem}

\begin{proof}
The fact that the nonzero coefficients alternate between $+1$ and $-1$ is due to Ozsv\'ath and Szab\'o \cite[Corollary~1.3]{osz-lens}.  The leading exponent is $g$ by a combination of \cite[Theorem~1.2]{osz-lens} and \cite[Theorem~1.2]{osz-genus}, while Hedden and Watson \cite[Corollary~9]{hedden-watson} proved that the next nonzero term is $-t^{g-1}$.
\end{proof}

\begin{corollary} \label{cor:alexander-roots-negative}
Let $K$ be an L-space knot of genus $g \geq 1$.  Then $\Delta_K(\lambda) > 0$ for any real number $\lambda > 1$.
\end{corollary}

\begin{proof}
By Theorem~\ref{thm:lspace-alexander}, we can group the nonzero monomials of $\Delta_K(t)$ as
\[ \Delta_K(t) = \left(\sum_{i=1}^k (t^{a_i} - t^{b_i})\right) + t^{-g}, \]
where $g = a_1 > b_1 > a_2 > b_2 > \dots > a_k > b_k = 1-g$.  Then $\lambda > 1$ implies that $\lambda^{a_i} > \lambda^{b_i}$ for each $i$, so the sum is strictly positive and $\Delta_K(\lambda) > \lambda^{-g}$.
\end{proof}

The following additional restriction on Alexander polynomials of L-space knots is a special case of a theorem of Krcatovich \cite{krcatovich}.

\begin{proposition} \label{prop:apply-krcatovich}
Let $K$ be an L-space knot of genus $g$.  Then its Alexander polynomial does not have the form
\[ \Delta_K(t) = \sum_{j=0}^{2k} (-1)^j t^{g-j} + O(t^{g-2k-2}), \]
where $t^{g-2k-1}$ is the highest-degree term below $t^g$ with coefficient zero, for any integer $k < g$.
\end{proposition}

\begin{proof}
Theorem~\ref{thm:lspace-alexander} implies that we can write $\Delta_K(t)$ in the form
\[ \Delta_K(t) = (1-t^{-1})\sum_{j=0}^\infty t^{a_j} = (1-t^{-1})(t^{a_0} + t^{a_1} + t^{a_2} + \dots) \]
for some strictly decreasing sequence $g = a_0 > a_1 > a_2 > \dots$ of integers.  Krcatovich \cite[Theorem~1.6]{krcatovich} proved that these integers must satisfy $a_j = -j$ for all $j \geq g$, and also $a_j \leq g - 2j$ for all $j < g$.  Now if $\Delta_K(t)$ had the form that we are trying to rule out, then this would mean that $a_j = g - 2j$ for all $j \leq k$, but also that
\[ a_{k+1} = g - 2k - 1 > g - 2(k+1), \]
which is impossible.
\end{proof}

We will study the monodromy of the fibration of an L-space knot complement, for which we need to know the following.

\begin{theorem}[{\cite{ghiggini,ni-hfk}}] \label{thm:lspace-fibered}
L-space knots are fibered.
\end{theorem}

\begin{theorem} \label{thm:lspace-veering}
L-space knots have right-veering monodromy.
\end{theorem}

\begin{proof}
The fibration of the complement of an L-space knot supports a contact structure on $S^3$, and if the supported contact structure is tight then Honda, Kazez, and Mati\'c \cite[Theorem~1.1]{hkm-veering} proved that the monodromy must be right-veering.  Now L-space knots are strongly quasipositive \cite[Corollary~1.4]{hedden-positivity}, which means that their associated open books support the tight contact structure on $S^3$ \cite[Proposition~2.1]{hedden-positivity}, and thus have right-veering monodromies as claimed.
\end{proof}

\begin{theorem}[\cite{ni-fixed-note,ghiggini-spano}] \label{thm:lspace-fixed}
Let $K$ be a hyperbolic L-space knot of genus $g > 1$, and suppose that its monodromy is freely isotopic to a pseudo-Anosov map
\[ \phi : \Sigma_{g,1} \to \Sigma_{g,1}. \]
Then $\phi$ has no fixed points.
\end{theorem}

\begin{remark}
Another route to Theorem~\ref{thm:lspace-veering} is to observe that L-space knots satisfy \[ \hfkhat(K,g-1) \cong \Z, \] because this portion of the knot Floer homology has rank at most 1 by \cite[Theorem~1.2]{osz-lens} but is nonzero by \cite[Corollary~9]{hedden-watson}.  A theorem of Ni \cite[Theorem~A.1]{ni-exceptional} now proves that the monodromy is veering, meaning either left-veering or right-veering.  

This argument applies more generally to so-called \emph{ffpf} knots, which are fibered knots satisfying \[ \hfkhat(K,g-1;\F_2) \cong \F_2, \] as in \cite[Lemma~2.8]{bs-traces}.  Hyperbolic ffpf knots also satisfy the conclusion of Theorem~\ref{thm:lspace-fixed}, namely that their monodromies have no fixed points \cite[Proposition~2.9]{bs-traces}.
\end{remark}

\section{Monodromies of knots with lens space surgeries} \label{sec:monodromy}

Let $K \subset S^3$ be a fibered hyperbolic knot of genus $g \geq 1$.  Then its monodromy is freely isotopic to a pseudo-Anosov homeomorphism
\[ \phi: \Sigma_{g,1} \to \Sigma_{g,1} \]
\cite{thurston-diffeomorphisms}.  As a pseudo-Anosov map, the monodromy $\phi$ preserves a transverse pair of transversely measured singular foliations $(\cF_s, \mu_s)$ and $(\cF_u, \mu_u)$, while contracting and expanding the respective measures $\mu_s$ and $\mu_u$ by the dilatation $\lambda(\phi) > 1$.  These foliations have some number $a_0 \geq 1$ of prongs at the puncture of $\Sigma_{g,1}$, as well as a collection of $k \geq 0$ interior singularities with varying numbers $a_1 \geq a_2 \geq \dots \geq a_k \geq 3$ of prongs; these numbers satisfy the Euler--Poincar\'e formula
\begin{equation} \label{eq:euler-poincare}
\sum_{i=0}^k (a_i - 2) = 4g - 4
\end{equation}
\cite[Proposition~5.1]{flp}.  We say that $\phi$ belongs to the stratum $(a_0; a_1,\dots,a_k)$.

\begin{lemma} \label{lem:even-prong-orientable}
Let $K$ be a fibered hyperbolic knot in a $\Z/2\Z$-homology sphere $Y$, and suppose that its pseudo-Anosov monodromy $\phi : \Sigma_{g,1} \to \Sigma_{g,1}$ belongs to a stratum $(a_0; a_1,\dots,a_k)$ where all of the $a_j$ are even.  Then the invariant foliations $\cF_s$ and $\cF_u$ of $\phi$ are both orientable.
\end{lemma}

\begin{proof}
Let $\cF$ denote one of the invariant foliations, with singular locus $S = \{p_1,\dots,p_k\}$.  Then the first Stiefel--Whitney class
\[ w_1\left(\cF|_{\Sigma_{g,1} \setminus S}\right) \]
vanishes on a small loop around each $p_i$, since the corresponding number $a_i$ of prongs is even.  It thus extends to an orientation class
\[ o(\cF) \in H^1(\Sigma_{g,1}; \Z/2\Z) \]
which is zero if and only if $\cF$ is orientable.  Assuming then that $\cF$ is non-orientable, this class is preserved by $\phi$ since $\cF$ is, so it is a nonzero element of
\[ \ker\left(1 - \phi^*: H^1(\Sigma_{g,1}; \Z/2\Z) \to H^1(\Sigma_{g,1}; \Z/2\Z)\right). \]
The Alexander polynomial $\Delta_K(t)$ is the characteristic polynomial of $\phi^*$, so it follows that
\[ \Delta_K(1) = \det(1-\phi^*) \equiv 0 \pmod{2}. \]
But then $|H_1(Y; \Z)| = |\Delta_K(1)|$ is even and this is a contradiction.
\end{proof}

\begin{lemma} \label{lem:even-prong-dilatation}
Let $K \subset S^3$ be a fibered hyperbolic knot, and suppose that its pseudo-Anosov monodromy $\phi$ belongs to a stratum $(a_0;a_1,\dots,a_k)$ where all of the $a_j$ are even.  If $\lambda > 1$ is the dilatation of $\phi$, then either $\lambda$ or $-\lambda$ is a root of the Alexander polynomial $\Delta_K(t)$.  If $K$ is also an L-space knot, then $\Delta_K(-\lambda) = 0$.  (In particular, the monodromy of an L-space knot reverses the orientation of its invariant foliations.)
\end{lemma}

\begin{proof}
Let $(\cF_u,\mu_u)$ be the unstable invariant foliation of $\phi$.  Lemma~\ref{lem:even-prong-orientable} guarantees that $\cF_u$ is orientable, hence the transverse measure $\mu_u$ defines a nonzero class $\alpha \in H^1(\Sigma_{g,1}; \R)$ such that
\[ \phi^*\alpha = \pm \lambda \alpha, \]
with the sign depending on whether $\phi$ preserves orientation or not.  Then $\pm\lambda$ is a root of $\det(t-\phi^*) = \Delta_K(t)$, as claimed.

Now if $K$ is an L-space knot then Corollary~\ref{cor:alexander-roots-negative} says that $\Delta_K(\lambda)$ is strictly positive, so in this case $-\lambda$ must be a root.
\end{proof}

We record the following consequence of Lemma~\ref{lem:even-prong-dilatation}, though we will not need it here.

\begin{proposition} \label{prop:lspace-finite-if-even}
For a fixed integer $g \geq 1$, there are only finitely many hyperbolic L-space knots of genus $g$ whose monodromies belong to strata of the form $(a_0;a_1,\dots,a_k)$ where all of the $a_j$ are even.
\end{proposition}

\begin{proof}
Let $K$ be such a knot, and let $\lambda > 1$ be the dilatation of its monodromy.  Theorem~\ref{thm:lspace-alexander} says that there are only finitely many possibilities for the Alexander polynomial $\Delta_K(t)$, and by Lemma~\ref{lem:even-prong-dilatation} we know that $-\lambda$ must be a root of $\Delta_K(t)$, so for fixed $g$ the dilatation $\lambda$ can only take one of finitely many values.  Let $C = C(g)$ be the largest such value.  Then up to conjugacy there are only finitely many pseudo-Anosov maps $\Sigma_{g,1} \to \Sigma_{g,1}$ with dilatation at most $C$ \cite{ivanov-bounded-dilatation}, and one of their mapping tori is homeomorphic to $S^3 \setminus N(K)$.  This means that there are only finitely many possibilities for the exterior of $K$, hence finitely many possible values of $K$ up to isotopy \cite{gordon-luecke-complement}.
\end{proof}

\begin{lemma} \label{lem:degeneracy}
Let $K \subset S^3$ be a hyperbolic L-space knot, with pseudo-Anosov monodromy $\phi$.  Then the pseudo-Anosov suspension flow of $\phi$ has degeneracy slope $\frac{a_0}{1}$, where $a_0 \geq 2$ is the number of boundary prongs of the invariant foliations for $\phi$.  Moreover, if some $r$-surgery on $K$ is reducible or has finite fundamental group, then $\Delta(\frac{a_0}{1}, r) \leq 1$.\footnote{Here $\Delta$ denotes the distance between two slopes on the torus, rather than an Alexander polynomial.}
\end{lemma}

\begin{proof}
Let $d$ be the degeneracy slope of the suspension flow.  Gabai and Oertel \cite[Theorem~5.3]{gabai-oertel} proved that any Dehn filling of $S^3 \setminus N(K)$ of slope $s$ admits an essential lamination unless $\Delta(d, s) \leq 1$.  Manifolds whose universal cover is not homeomorphic to $\R^3$ do not admit essential laminations \cite[Theorem~6.1]{gabai-oertel}, though, so we can apply this to the meridional slope $\frac{1}{0}$ (whose filling is $S^3$) to get $\Delta(d, \frac{1}{0}) \leq 1$, and to any reducible or finite filling of slope $s = r$ to get $\Delta(d, r) \leq 1$.  

It remains to be seen that $d = \frac{a_0}{1}$, and that $a_0 \geq 2$.  The fact that $\Delta(d,\frac{1}{0}) \leq 1$ tells us that either $d = \frac{1}{0}$, or $d = \pm \frac{a_0}{1}$ since the numerator must equal the number of boundary prongs.  As the monodromy of an L-space knot, the map $\phi$ is right-veering by Theorem~\ref{thm:lspace-veering}, so the fractional Dehn twist coefficient $c(\phi) = \frac{1}{d}$ is strictly positive \cite[Proposition~3.1]{hkm-veering}, hence $d = \frac{a_0}{1}$.  The claim that $a_0 \geq 2$ follows, because otherwise $c(\phi) = \frac{1}{a_0} \geq 1$ would imply by \cite[Theorem~1.2]{hkm-rv2} that $S^3$ supports a taut foliation, which is absurd.
\end{proof}

Lemma~\ref{lem:degeneracy} leads to a quick proof of the upper bound in Conjecture~\ref{conj:goda-teragaito}.

\begin{proposition} \label{prop:upper-bound}
Let $K \subset S^3$ be a hyperbolic knot of genus $g \geq 1$, and suppose that $r$-surgery on $K$ is an elliptic 3-manifold for some $r > 0$.  Then $r \leq 4g-1$.
\end{proposition}

\begin{proof}
The surgery in question is an L-space by Theorem~\ref{thm:elliptic}, so Theorem~\ref{thm:lspace-fibered} guarantees that $K$ is fibered.  Let $\phi$ be the pseudo-Anosov map isotopic to the monodromy of $K$, and $n = a_0$ the number of boundary prongs for the invariant foliations of $\phi$.  Lemma~\ref{lem:degeneracy} says that $\Delta(n, r) \leq 1$, whence
\[ r \leq n + 1. \]
But the Euler--Poincar\'e formula \eqref{eq:euler-poincare} tells us that $n - 2 \leq 4g - 4$, so $n + 1 \leq 4g-1$ and the claimed inequality follows.
\end{proof}

The case of Proposition~\ref{prop:upper-bound} where $(4g-1)$-surgery on $K$ is a lens space is already well-understood: Baker \cite[Theorem~1.2]{baker-smallgenus} proved that $K$ must then be a Berge knot.

\begin{lemma} \label{lem:stratum-4g-2}
Let $K \subset S^3$ be a hyperbolic knot of genus $g \geq 1$, and suppose for some $r > 4g-3$ that $r$-surgery on $K$ is a lens space, or more generally a 3-manifold with elliptic geometry.  Then the monodromy $\phi$ of $K$ belongs to the stratum $(4g-2;)$.  In other words, the invariant foliations of $\phi$ have $4g-2$ boundary prongs and no interior singularities.
\end{lemma}

\begin{proof}
Let $a_0$ be the number of boundary prongs.  Lemma~\ref{lem:degeneracy} says that $\Delta(a_0,r) \geq 1$, so
\[ a_0 \geq r - 1 > 4g - 4. \]
Then $a_0 \geq 4g-3$, so by \eqref{eq:euler-poincare}, the only possible strata for $\phi$ are
\[ (a_0; a_1,\dots,a_k) = (4g-2;) \text{ or } (4g-3;3). \]
In the latter case, since $\phi$ fixes the stable foliation $\cF_s$ it must permute any 3-pronged singularities, and this means that the unique $3$-pronged singularity would be a fixed point of $\phi$.  This contradicts the assertion of Theorem~\ref{thm:lspace-fixed} that L-space knots have fixed point free monodromy, so the only possible stratum is $(4g-2;)$.
\end{proof}

\begin{corollary} \label{cor:alexander-dilatation}
Let $K \subset S^3$ be a hyperbolic knot of genus $g \geq 1$, and suppose that $r$-surgery on $K$ has elliptic geometry for some $r > 4g-3$.  Let $\lambda > 1$ be the dilatation of the monodromy of $K$.  Then $-\lambda$ is a root of the Alexander polynomial $\Delta_K(t)$.
\end{corollary}

\begin{proof}
Lemma~\ref{lem:stratum-4g-2} says that the monodromy belongs to the stratum $(4g-2;)$, in which every prong number is even, so then Lemma~\ref{lem:even-prong-dilatation} applies.
\end{proof}

\begin{example} \label{ex:p-239}
The genus-6 pretzel knot $P(-2,3,9)$ is hyperbolic, and Bleiler and Hodgson \cite[\S5]{bleiler-hodgson} showed that it has a pair of non-cyclic elliptic surgeries of slopes $22$ and $23$.  Both of these slopes exceed $4g-3 = 21$, so Lemma~\ref{lem:stratum-4g-2} says that the monodromy of $P(-2,3,9)$ is in the stratum $(22;)$.  Corollary~\ref{cor:alexander-dilatation} also applies in both cases to say that if the monodromy has dilatation $\lambda$, then $-\lambda < -1$ is a root of
\begin{align*}
\Delta_{P(-2,3,9)}(t) &= t^6 - t^5 + t^3 - t^2 + t -1 + t^{-1} - t^{-2} + t^{-3} - t^{-5} + t^{-6} \\
&= (t - 1 + t^{-1})(t^5 - t^3 + 1 - t^{-3} + t^{-5}).
\end{align*}
This has exactly one real root less than $-1$, so we must have $\lambda \approx 1.26123$.
\end{example}

\begin{lemma} \label{lem:genus-4-cyclotomic}
Let $K$ be an L-space knot of genus at most $4$.  Then all of the roots of $\Delta_K(t)$ lie on the unit circle.
\end{lemma}

\begin{proof}
We enumerate all of the Alexander polynomials permitted by Theorem~\ref{thm:lspace-alexander}; for a fixed genus $g \geq 2$ there are at most $2^{g-2}$ possibilities, determined completely by which of the monomials $t^{g-2}, t^{g-1},\dots,t^1$ have nonzero coefficients.  For $g \leq 3$ we have only
\begin{align*}
\Delta_{T_{2,3}}(t) &= t - 1 + t^{-1} \\
\Delta_{T_{2,5}}(t) &= t^2 - t + 1 - t^{-1} + t^{-2} \\
\Delta_{T_{2,7}}(t) &= t^3 - t^2 + t - 1 + t^{-1} - t^{-2} + t^{-3} \\
\Delta_{C_{3,2}(T_{2,3})}(t) &= t^3 - t^2 + 1 - t^{-2} + t^{-3} \\
&= (t - 1 + t^{-1})(t^2 - 1 + t^{-2}).
\end{align*}
All four of these have all of their roots on the unit circle.  When $g=4$ three out of the four possibilities are
\begin{align*}
\Delta_{T_{2,9}}(t) &= t^4 - t^3 + t^2 - t + 1 - t^{-1} + t^{-2} - t^{-3} + t^{-4} \\
\Delta_{T_{3,5}}(t) &= t^4 - t^3 + t - 1 + t^{-1} - t^{-3} + t^{-4} \\
\Delta_{C_{5,2}(T_{2,3})}(t) &= t^4 - t^3 + 1 - t^{-3} + t^{-4} \\
&= (t^2 - t + 1 - t^{-1} + t^{-2})(t^2 - 1 + t^{-2}),
\end{align*}
all of whose roots lie on the unit circle; and the fourth is
\[ t^4 - t^3 + t^2 - 1 + t^{-2} - t^{-3} + t^{-4}, \]
which cannot happen by Proposition~\ref{prop:apply-krcatovich}.
\end{proof}

In genus $5$, Theorem~\ref{thm:lspace-alexander} gives us eight possible Alexander polynomials, which we can reduce to six by eliminating the two excluded by Proposition~\ref{prop:apply-krcatovich}.  Three of the six are
\begin{equation} \label{eq:polynomials-5-cyclotomic}
\begin{aligned}
\Delta_{T_{2,11}}(t) &= t^5 - t^4 + t^3 - t^2 + t - 1 + t^{-1} - t^{-2} + t^{-3} - t^{-4} + t^{-5} \\
\Delta_{C_{7,2}(T_{2,3})}(t) &= t^5 - t^4 + t - 1 + t^{-1} - t^{-4} + t^{-5} \\
f_1(t) &= t^5 - t^4 + t^2 - 1 + t^{-2} - t^{-4} + t^{-5} \\
&= (t - 1 + t^{-1})(t^2 - t + 1 - t^{-1} + t^{-2}),
\end{aligned}
\end{equation}
all of whose roots lie on the unit circle, and the remaining ones are
\begin{equation} \label{eq:polynomials-5-not}
\begin{aligned}
\Delta_{P(-2,3,7)}(t) &= t^5 - t^4 + t^2 - t + 1 - t^{-1} + t^{-2} - t^{-4} + t^{-5} \\
f_2(t) &= t^5 - t^4 + t^3 - t^2 + 1 - t^{-2} + t^{-3} - t^{-4} + t^{-5} \\
f_3(t) &= t^5 - t^4 + 1 - t^{-4} + t^{-5}.
\end{aligned}
\end{equation}

We can now conclude the proof of Theorem~\ref{thm:large-elliptic-surgery-main}, which says that if $K$ is hyperbolic of genus $g \leq 5$, and some surgery on $K$ of slope greater than $4g-3$ produces an elliptic manifold, then $K = P(-2,3,7)$.

\begin{proof}[Proof of Theorem~\ref{thm:large-elliptic-surgery-main}]
Let $\phi : \Sigma_{g,1} \to \Sigma_{g,1}$ be the monodromy of $K$, with dilatation $\lambda > 1$.  Corollary~\ref{cor:alexander-dilatation} says that $-\lambda$ is a root of the Alexander polynomial $\Delta_K(t)$.  Then $g \leq 4$ is ruled out by Lemma~\ref{lem:genus-4-cyclotomic}, so we must have $g = 5$.  By the same reasoning $\Delta_K(t)$ cannot possibly be any of the polynomials in \eqref{eq:polynomials-5-cyclotomic}, so it must be one of the polynomials appearing in \eqref{eq:polynomials-5-not}.  In fact, since $\lambda > 1$ we have
\[ f_2(-\lambda) < -\lambda^5 - \lambda^4 - \lambda^3 - \lambda^2 + 1 < -4 + 1 = -3, \]
and likewise $f_3(-\lambda) < -1$, so the only possibility is
\[ \Delta_K(t) = \Delta_{P(-2,3,7)}(t). \]
The dilatation $\lambda$ is then the unique real root of $\Delta_K(-t)$ greater than $1$, which is Lehmer's number $1.17628\ldots$.

Tsang \cite[Theorem~1.4]{tsang-dilatation} classified all of the fully-punctured pseudo-Anosov maps with normalized dilatation less than $\mu^4 \approx 6.854$, where $\mu = \frac{1}{2}(1+\sqrt{5})$ is the golden ratio.  The normalized dilatation of a pseudo-Anosov map $\psi: \Sigma \to \Sigma$ with dilatation $\lambda$ is defined to be $\lambda^{|\chi(\Sigma)|}$.  In our case, the invariant foliations of $\phi$ have no interior singularities by Lemma~\ref{lem:stratum-4g-2}, so $\phi$ is already fully punctured and the normalized dilatation is $\lambda(\phi)^{|\chi(\Sigma_{g,1})|} = \lambda^9 \approx 4.311$, which is less than $\mu^4$.  Then \cite[Table~1]{tsang-dilatation} says that this normalized dilatation is uniquely achieved by the monodromy $\phi_P$ of $P(-2,3,7)$, so $\phi$ must be conjugate to $\phi_P$.  This means that their mapping tori $S^3 \setminus N(K)$ and $S^3 \setminus N\big(P(-2,3,7)\big)$ are homeomorphic, so $K = P(-2,3,7)$ after all \cite{gordon-luecke-complement}.
\end{proof}

\bibliographystyle{myalpha}
\bibliography{References}

\end{document}